\documentclass[12pt,reqno]{amsart}
\usepackage{amssymb,latexsym}
\usepackage{amsmath,color}
\usepackage{amsthm}
\usepackage{epsfig}
\usepackage{graphicx}
\usepackage{titletoc}
\usepackage{dsfont}
\usepackage{nccbbb}
\usepackage{amsmath}
\usepackage{amssymb}
\usepackage{amssymb}
\usepackage{amstext}
\usepackage{amsopn}
\usepackage{amsbsy}
\usepackage{amscd}
\usepackage{amsxtra}
\usepackage{accents}
\usepackage{bbm}
\usepackage{yfonts}
\usepackage{bbm}
\usepackage[pagewise]{lineno}
\usepackage[dvipsnames]{xcolor} 
\definecolor{darkgreen}{rgb}{0.0, 0.5, 0.0}

\usepackage[colorlinks=true, linkcolor=blue, citecolor=magenta, urlcolor=cyan]{hyperref}

\numberwithin{equation}{section}

\usepackage{xparse}
\usepackage{microtype}

        \makeatletter
        \newcommand{\ovset}[3][0ex]{%
          \mathrel{\mathop{#3}\limits^{
            \vbox to#1{\kern0\ex@
            \hbox{$\scriptstyle#2$}\vss}}}}
        \makeatother

\NewDocumentCommand{\hideparbox}{O{c}mm+m}
 {
  \group_begin:
  \vbox_set:Nn \l_hideparbox_box
   {
    \use:c { @parboxrestore }
    \hsize=#3\scan_stop:
    \strut#4\par
   }
  \vbadness=\c_ten_thousand 
  \vbox_set_split_to_ht:NNn \l_hideparbox_box \l_hideparbox_box { #2 }
  \parbox[#1][#2]{#3}
   {
    \vbox_unpack:N \l_hideparbox_box
   }
  \group_end:
 }
\ExplSyntaxOff

\usepackage{cases}

\usepackage[T1]{fontenc}

\newtheorem{theorem}{Theorem}[section]
\newtheorem{proposition}[theorem]{Proposition}
\newtheorem{lemma}[theorem]{Lemma}
\newtheorem{corollary}[theorem]{Corollary}
\newtheorem{remark}[theorem]{Remark}

\theoremstyle{definition}

\newtheorem{example}[theorem]{Example}

\newcommand{\ds}{\displaystyle}

\newcommand{\dss}{\mathrm{d}s}

\newcommand{\ue}{u_\lambda}

\newcommand{\ve}{v_{\lambda,\mu}}
\newcommand{\dx}{\,\mathrm{d}x}
\newcommand{\dy}{\,\mathrm{d}y}

\usepackage[misc]{ifsym}
\allowdisplaybreaks

\title[Elliptic equations with nonlocal boundary conditions]{}

\address{\rm (Chiun-Chang Lee) Institute for Computational and Modeling Science, National Tsing Hua University, Hsinchu 300, Taiwan}
\email{lee2@mx.nthu.edu.tw}

\begin{document}

\maketitle
\vspace*{-66pt}
\begin{center}
{\large\bf An existence theorem for elliptic equations with nonlocal boundary conditions}\vspace*{8pt}\\
 Chiun-Chang Lee
\end{center}
\begin{abstract} 
{\scriptsize 
 The focus of this study is on exploring some qualitative properties of solutions to a class of semilinear elliptic problems in bounded domains, where the boundary conditions depend non-locally on the unknown solution at specified interior points and its integral. The primary approach integrates a fixed-point argument with refined asymptotic estimates to establish the existence and structure of solutions. Furthermore, the maximum principles are established under practical nonlocal-type boundary conditions.
\vspace{3pt}
}
\end{abstract}

\section{\bf  Introduction} 

{\color{black} Significant attention has been focused on a class of semilinear equations defined in bounded domains. In these equations, the boundary conditions incorporate nonlocal terms that depend on the unknown solution at specific interior points (referred to as the multi-point boundary condition) or involve the integral of the unknown solution over a subdomain (referred to as the integral boundary condition).} The corresponding problems often arise in practical applications. For example, the multi-point boundary condition is used in the design of bridges and highways, to model the displacement of the bridge from its unloaded position ({\color{black} see, e.g.,} \cite{ZHZ2007}). Various nonlocal type boundary conditions can also be found in linear optimal control theory and stochastic multi-point boundary value problems. We refer the reader to \cite{AM1950,BF1980,FD2017,G2012,TD2012,W1942} for the related models. For further detailed background of the models with numerical algorithms and applications, see, e.g., \cite{CLW2025,IM1987-1,M1987,N2006,R1974,U1966} and references therein.

Under such nonlocal boundary conditions, the equation no longer possesses a
variational structure. To the best of our knowledge, most of the previous references have concentrated on one-dimensional equations with linear-type multi-point boundary conditions or integral boundary conditions. (We will continue the discussion after formulating the main problem.) However, comprehensive theoretical results for high-dimensional semilinear equations with nonlocal boundary conditions appear to be scarce.

Several open problems persist due to the intricate relationship between the interior and boundary behavior of unknown solutions in nonlocal boundary conditions.
 The present study will explore semilinear problems with various nonlocal boundary conditions. Our approach is novel and allows us to study a wide range of nonlocal boundary conditions. {\color{black} To} be specific, {\color{black} the primary approach combines fixed-point arguments with asymptotic analysis. The method allows us to establish sufficient conditions for the existence and uniqueness of solutions. This framework represents a significant feature and highlight of this work.}

\subsection*{Position of the problem} Let $\lambda$ be a positive parameter, and $\nabla$ represent the standard gradient operator in $\mathbb{R}^N$, $N\geq1$. The equation is formulated  as
\begin{align}\label{equ}
-\nabla\cdot(D(x)\nabla\ue(x))+{\lambda}f(x,\ue(x))=0,&\quad\,\,\,x\in\Omega, 
\end{align}
 subject to the boundary condition 
\begin{equation}\label{bdu}
\ue(x)=g(x)+\mathfrak{B}\big(\ue(\xi_1),...,\ue(\xi_m),\int_{\Omega}w(y)\Phi(\ue(y))\text{d}y\big),\quad\,x\in\partial\Omega,
\end{equation}
where 
$\Omega\subset\mathbb{R}^N$ is a bounded domain  with  a smooth boundary~$\partial\Omega$, and $\xi_1$, ..., $\xi_m$ are {\color{black} $m$ given}  distinct points in $\Omega$. In what follows, we assume that $w\in\text{L}^1(\Omega)$, $D\in\text{C}^{1,\tau}(\overline{\Omega};(0,\infty))$ is positive,  $g\in\text{C}^{0,\tau}(\partial\Omega;\mathbb{R})$, and $\mathfrak{B}:\mathbb{R}^{m+1}\to\mathbb{R}$ and $\Phi:\mathbb{R}\to\mathbb{R}$  are locally Lipschitz continuous in each variable with the Lipschitz constant $\ell>0$ (see \eqref{lip-b} below). Here $\text{C}^{k,\tau}$ denotes the space of functions with H\"{o}lder continuous derivatives up to order $k$ and a H\"{o}lder exponent $\tau\in(0,1)$.  The nonlinear source term~$f:\overline{\Omega}\times\mathbb{R}\to\mathbb{R}$
satisfies certain conditions
which will be specified after we review the related scientific background of~\eqref{equ}--\eqref{bdu}. 

\subsection*{Scientific background} Nonlocal boundary value problems have a rich history in the fields of modeling and mathematical analysis. Approximately fifty years ago, Keller  in his work \cite{K1969}  considered a class of one-dimensional models with solution $u$ satisfying a constraint $\sum_{j=1}^{\widetilde{m}}\beta_ju(\widetilde{\xi}_j)+\int_{\Omega}w(y)u(y)\,\text{d}y=\beta_{\widetilde{m}+1}$ for some $\widetilde{m}\in\mathbb{N}$, where $\beta_j$'s are non-zero constants.  Since then, the study of multi-point and integral constraints on the boundary has garnered significant attention due to its widespread applications in various fields of applied sciences~\cite{BK1974,GNA1971,LM2023,PS2018}. In these applications, the nonlocal term $\mathfrak{B}$ in \eqref{bdu} exhibits nonlocal dependencies on the unknowns $\ue(\xi_j)$ and $\int_{\Omega}w(y)\ue(y)\text{d}y$. We will now delve into the background of \eqref{equ}--\eqref{bdu} and underline the significance of our work.

 We distinguish two types within this category: linear type multi-point boundary condition $\mathfrak{B}=\sum_{j=1}^m\beta_j\ue(\xi_j)$ with nonzero constants $\beta_j$'s, and linear type integral boundary condition~$\mathfrak{B}=\int_{\Omega}w(y)\ue(y)\text{d}y$, where the former one is known as the multi-point boundary condition~({\color{black} see} \cite[Sections~2--3]{AR1981} and \cite{BGK1973,K1969}), and the latter one is referred to the integral boundary condition ({\color{black} see} \cite{CLM2024,L2023}). In the past few decades, the related problems have turned into an extremely active field of research in nonlinear elliptic equations~\eqref{equ}--\eqref{bdu} for specific cases of boundary conditions, such as when $\mathfrak{B}=\sum_{j=1}^m\beta_j\ue(\xi_j)$ with $\beta_j>0$ and $\sum_{j=1}^m\beta_j<1$, or when $\mathfrak{B}=\int_{\Omega}w(y)\ue(y)\text{d}y$ with $w\geq0$ and $\int_{\Omega}w(y)\text{d}y<1$. In these cases, the maximum principle can be applied to estimate the boundary value of $\ue$. {\color{black} However, it is worth noting that for general nonlocal boundary conditions, such as those described in \eqref{bdu}, the related issue presents a significant challenge. This is because the maximum principle is typically not applicable in these cases.}

It is expected that the existence and asymptotics (e.g., with respect to $\lambda\gg1$) of solution $\ue$ to \eqref{equ}--\eqref{bdu} may depend on the relation among the parameter $\lambda$, the form of $\mathfrak{B}\big(\ue(\xi_1),...,\ue(\xi_m),\int_{\Omega}w(y)\Phi(\ue(y))\text{d}y\big)$ and the property of the nonlinear source $f$.  {\color{black} To give the reader a concrete understanding of this work, we first consider a one-dimensional linear equation of the form \eqref{equ}--\eqref{bdu} in Section~\ref{sec-ov}. Specifically, we analyze the case with a multi-point boundary condition $\mathfrak{B}=\sum_{j=1}^m\beta_j\ue(\xi_j)$ and demonstrate that there exists a $\lambda>0$ (depending on $\beta_j$ and $\xi_j$) for which this equation has no solution.}

Recognizing the general invalidity of the maximum principle due to nonlocal boundary conditions, the study of semilinear problems \eqref{equ}--\eqref{bdu} continues to evolve. It should be emphasized that the existence issue for \eqref{equ}--\eqref{bdu} with respect to $\lambda>0$ is, of course, much more difficult. Accordingly,  this work will establish a sufficient condition of $\lambda$  such that \eqref{equ}--\eqref{bdu} has a solution $\ue$. Moreover, when $\mathfrak{B}:\mathbb{R}^{m+1}\to\mathbb{R}$ and $\Phi:\mathbb{R}\to\mathbb{R}$  are globally Lipschitz continuous, we verify the uniqueness for equation~\eqref{equ}--\eqref{bdu} with sufficiently large $\lambda$ ({\color{black} see} Theorem~\ref{m-thm}(I)) and corresponding maximum principles~({\color{black} see} Corollary~\ref{cor1}), where the asymptotic analysis with respect to $\lambda\gg1$ plays a crucial role in this study. We further explore the existence of the semilinear problem~\eqref{equ}--\eqref{bdu} when at least one of $\mathfrak{B}:\mathbb{R}^{m+1}\to\mathbb{R}$ or $\Phi:\mathbb{R}\to\mathbb{R}$  is only locally Lipschitz continuous. Among these investigations, it is noteworthy that without the refined asymptotics associated with $\lambda\gg1$, the construction of a contraction mapping (see \eqref{map-t}) for proving the existence and uniqueness would not be feasible ({\color{black} see} Proposition~\ref{prop-v} and Lemma~\ref{lem-t}). To the best of our understanding, however, such ideas are relatively scarce in the recent mathematical literature.

Finally, we want to point out that when $m=\infty$, the boundary condition~\eqref{bdu} includes the infinite-point boundary condition~\cite{EA2020}. In this situation, it becomes essential to consider appropriate conditions to ensure the well-defined nature of $\mathcal{B}$. The exploration of this particular issue will be undertaken in our forthcoming work~\cite{L2024}.\\
\noindent
{\bf\color{black} Outline.} To introduce the main results including the existence, uniqueness and maximum principle, the structure of this paper unfolds systematically, step by step. The remainder of this paper is organized into two sections. Section~\ref{sec-ov} presents illustrative examples highlighting  situations involving non-existence, uniqueness, and multiplicity of solutions for equation~\eqref{equ}--\eqref{bdu}. Based on these examples, specific assumptions for $f$, $\mathfrak{B}$, and $\Phi$ are formulated in Section~\ref{sec-os}.  To establish the existence theorem for equation~\eqref{equ}--\eqref{bdu}, we introduce a mapping in Section~\ref{sec-os} that corresponds to the boundary condition~\eqref{bdu}. The approach relies on the fixed-point argument and involves deriving a priori estimates for the solution of a corresponding local-type equation concerning $\lambda\gg1$. The main result will be articulated in Section~\ref{sec-fa}.

\section{\bf Heuristics and Preliminaries}\label{sec-ov}

In general situations, literature provides limited results concerning the existence of solutions to \eqref{equ}--\eqref{bdu}. Let us present some illustrative examples before delving into the theoretical results. It must be emphasized that although our examples focus on one-dimensional equations, 
{\color{black} the insights gained from them provide a valuable foundation for exploring the existence of solutions to high-dimensional equations.}

\subsection{\bf Examples for the non-existence, uniqueness and multiplicity} 
Specifically, we first consider the case $N=1$ and an equation
 \begin{equation}\label{1u-eq}
 -\ue''(x)+\lambda{\ue(x)}=0,\,\,x\in(L,R),    
 \end{equation}
  with the boundary condition corresponding to \eqref{bdu}:
\begin{equation}\label{1-Db}
    \ue(L)=0,\quad\ue(R)=g_R+\sum_{j=1}^{m}\beta_{j}\ue(\widetilde{\xi}_{j}),
\end{equation}
where  $g_R$ is a constant and $\widetilde{\xi}_{j}\in(L,R)$ with  $\widetilde{\xi}_{j}<\widetilde{\xi}_{j+1}$. We provide a practical example for the effect of $\beta_j$'s and $\widetilde{\xi}_{j}$'s on the existence and uniqueness/multiplicity of solutions. 
\begin{example}\label{ex1}
 Define
  \begin{equation*}
 \eta(s)=\text{e}^{s(R-L)}-\text{e}^{-s(R-L)}-\sum_{j=1}^m\beta_j\left(\text{e}^{s(\widetilde{\xi}_{j}-L)}-\text{e}^{-s(\widetilde{\xi}_{j}-L)}\right)     
  \end{equation*}
  and a set 
\begin{equation}\label{sel}
 \mathfrak{S}_{\eta}:=\left\{\lambda>0:\,\eta(\sqrt{\lambda})\neq0\right\}.   
\end{equation}
 When $\lambda\in\mathfrak{S}_{\eta}$,  equation \eqref{1u-eq}--\eqref{1-Db} has a unique solution
\begin{equation}\label{sou}
    \ue(x)=\frac{g_R}{\eta(\sqrt{\lambda})}\left(\text{e}^{\sqrt{\lambda}(x-L)}-\text{e}^{-\sqrt{\lambda}(x-L)}\right).
\end{equation}
However, when $\lambda\not\in\mathfrak{S}_{\eta}$, i.e., $\eta(\sqrt{\lambda})=0$, the multiplicity and non-existence of $\ue$ may depend on  $g_R$:
\begin{itemize}
    \item[$\blacksquare$] If $g_R=0$, this equation has infinitely many solutions since for arbitrary $b^*$, $\ue^*(x)=b^*\left(\text{e}^{\sqrt{\lambda}(x-L)}-\text{e}^{-\sqrt{\lambda}(x-L)}\right)$
   is a solution of \eqref{1u-eq} with the boundary condition~\eqref{1-Db}.
     \item[$\blacksquare$] If $g_R\neq0$, then \eqref{1u-eq} with the boundary condition~\eqref{1-Db} has no solution.
     \end{itemize}
Besides, for arbitrary $\boldsymbol{\beta}=(\beta_1,...,\beta_m)$ and $\boldsymbol{\xi}_R=(\xi_{1,R},...,\xi_{m,R})$, we have
\begin{equation}\label{iql}
\left[\lambda^*(\boldsymbol{\beta},\boldsymbol{\xi}_R),\infty\right)\subset\mathfrak{S}_{\eta},\quad
\lambda^*(\boldsymbol{\beta},\boldsymbol{\xi}_R)=\left(\frac{1}{R-\xi_{R,m}}\log\max\left\{\mathcal{B},\frac{1}{\mathcal{B}}\right\}\right)^2,
\end{equation}
where $\mathcal{B}=\sum_{j=1}^m|\beta_j|>0$.\footnote{Note that $\lambda\geq\lambda^*(\boldsymbol{\beta},\boldsymbol{\xi}_R)$ implies $\text{e}^{\sqrt{\lambda}(R-L)}>\sum_{j=1}^m|\beta_j|\text{e}^{\sqrt{\lambda}(\widetilde{\xi}_{j}-L)}$ and $\text{e}^{-\sqrt{\lambda}(R-L)}<\sum_{j=1}^m|\beta_j|\text{e}^{-\sqrt{\lambda}(\widetilde{\xi}_{j}-L)}$ which guarantees  $\eta(\sqrt{\lambda})>0$.}
We shall stress that the structure of $\mathfrak{S}_{\eta}$ depends on $\beta_j$'s and $\widetilde{\xi}_{j}$'s, which {\color{black}is} introduced as follows.

\begin{itemize}
 \item[$\blacksquare$] For the case $  \sum_{j=1}^m\max\{\beta_j,0\}(\widetilde{\xi}_{j}-L)\leq{R-L}$,
   we have $\mathfrak{S}_{\eta}=(0,\infty)$,\footnote{
Note that $\eta(0)=0$ and $\eta'(s)=(R-L)(\text{e}^{s(R-L)}+\text{e}^{-s(R-L)})-\sum_{j=1}^m\beta_j(\widetilde{\xi}_{j}-L)(\text{e}^{s(\widetilde{\xi}_{j}-L)}+\text{e}^{-s(\widetilde{\xi}_{j}-L)})>{\left(R-L-\sum_{j=1}^m\max\{\beta_j,0\}(\widetilde{\xi}_{j}-L))\right)}(\text{e}^{s(R-L)}+\text{e}^{-s(R-L)})\geq0$ since $0<\widetilde{\xi}_{j}-L<R-L$. Hence, $\eta(\sqrt{\lambda})>0$ for $\lambda>0$.} which results in that for each $\lambda>0$, \eqref{1u-eq} with the boundary condition~\eqref{1-Db} has a unique solution~\eqref{sou}.
  \item[$\blacksquare$] If $\sum_{j=1}^m\beta_j(\widetilde{\xi}_{j}-L)>{R-L}$ (note that $\beta_j$ is not necessarily positive), then $\mathfrak{S}_{\eta}\neq(0,\infty)$.\footnote{Note that $\eta(0)=0$, $\eta'(0)=2{\left(R-L-\sum_{j=1}^m\beta_j(\widetilde{\xi}_{j}-L))\right)}<0$ and $\lim_{s\to\infty}\eta(s)=\infty$. This implies that $\eta(\sqrt{\lambda})=0$ has at least one positive root.}
    \end{itemize}
\end{example}
Example~\ref{ex1} shows that for given $\beta_j\neq0$ and distinct points $\widetilde{\xi}_{j}\in(L,R)$, $j=1,...,m$ (without any other constraint), \eqref{sel} and \eqref{iql} ensure the existence and uniqueness of \eqref{1u-eq} with the boundary condition~\eqref{1-Db} as $\lambda>0$ is sufficiently large. On the other hand, it is worth stressing that when either $\mathfrak{B}:\mathbb{R}^{m+1}\to\mathbb{R}$ or $\Phi:\mathbb{R}\to\mathbb{R}$ is locally Lipschitz continuous, \eqref{equ}--\eqref{bdu} may admit multiple solutions. The following example focuses on the case of integral boundary conditions. 
\begin{example}\label{ex2}
For simplicity, we let $(L, R) = (0, 1)$ in \eqref{1u-eq}. We consider equation \eqref{1u-eq} with the boundary conditions $\ue(0) = 0$ and $\ue(1) = \int_0^1 \ue$. Then, for each $\lambda > 0$, it can be readily verified that it possesses a unique solution $\ue \equiv 0$. This case corresponds to \eqref{bdu} with both globally Lipschitz continuous functions $\mathfrak{B}:\mathbb{R}^{m+1}\to\mathbb{R}$ and $\Phi:\mathbb{R}\to\mathbb{R}$ ({\color{black} see \eqref{lip-b}} in Section~\ref{sec-os}). However, if either $\mathfrak{B}$ or $\Phi$ is locally Lipschitz continuous, the uniqueness may no longer be guaranteed.
\begin{itemize}
 \item[$\blacksquare$] \eqref{1u-eq} with the boundary conditions
\begin{equation*}
\ue(0)=0\quad\text{and}\quad\ue(1)=\sqrt{\int_0^1|\ue|\,\text{d}y}  
 \end{equation*}
 has two solutions $\ue\equiv0$ and 
 \begin{equation*}
    \ue(x)=\frac{\sinh(\sqrt{\lambda}x)}{2\sqrt{\lambda}\cosh^2\frac{\sqrt{\lambda}}{2}}\xrightarrow{\lambda\to\infty}0\,\,\text{uniformly\,\,in}\,\,[0,1]. 
 \end{equation*}
As a consequence, this equation admits two distinct solutions, both of which tends to zero uniformly on $[0,1]$ when $\lambda\gg1$.
 \item[$\blacksquare$]  {\color{black} We observe that the equation \eqref{1u-eq} subject to the boundary conditions
\begin{equation}\label{exbd1}
\ue(0)=0\quad\text{and}\quad\ue(1)=\int_0^1\ue^2(y)\,\text{d}y  
 \end{equation}
 admits two distinct solutions, one being the trivial solution $\ue\equiv0$ and the other being the nontrivial solution  
 \begin{equation}\label{2025u}
\ue(x)=\frac{2\sinh\sqrt{\lambda}}{\frac{\sinh(2\sqrt{\lambda})}{2\sqrt{\lambda}}-1}\sinh(\sqrt{\lambda}x).    \end{equation}
Moreover, let $0<\kappa<1$ be a constant independent of $\lambda$. Then, \eqref{2025u} satisfies~$\lim_{\lambda\to\infty}\max_{[0,1-\kappa]}|\ue|\xrightarrow{\lambda\to\infty}0$, while its boundary asymptotics at $x=1$ satisfies $\lim_{\lambda\to\infty}\frac{\ue(1)}{2\sqrt{\lambda}}=1$. As a consequence, under the integral boundary conditions \eqref{exbd1}, this equation admits two solutions. One of these solutions remains uniformly bounded on $[0,1]$, whereas the other exhibits asymptotic blow-up as $\lambda\to\infty$.}
\end{itemize}
\end{example}

Based on the insights gained from Examples \ref{ex1} and \ref{ex2}, we consider a class of nonlinear sources $f$ introduced in Section~\ref{sec-os}. Our objective is to establish the existence of solutions to \eqref{equ}--\eqref{bdu} for sufficiently large $\lambda>0$. Of particular interest is the case when both $\mathfrak{B}:\mathbb{R}^{m+1}\to\mathbb{R}$ and $\Phi:\mathbb{R}\to\mathbb{R}$ are globally Lipschitz continuous. In this situation, we will prove that for sufficiently large $\lambda$, \eqref{equ}--\eqref{bdu} possesses a unique solution. This result stands in contrast to situations where either $\mathfrak{B}:\mathbb{R}^{m+1}\to\mathbb{R}$ or $\Phi:\mathbb{R}\to\mathbb{R}$ is only locally Lipschitz continuous.

\subsection{\bf Overview and strategy}\label{sec-os}
In what follows, we shall set
\begin{equation}\label{mxi}
 \mathfrak{m}(\boldsymbol{\xi}):=\min_{1\leq {j}\leq {m}}\mathrm{dist}(\xi_j,\partial\Omega)>0\quad(\boldsymbol{\xi}=(\xi_1,...,\xi_m)).  
\end{equation}
For $0<\delta<\mathrm{diam}(\Omega)$, we define 
\begin{equation}\label{i-O}
    \Omega_{\delta}:=\{x\in\Omega:\mathrm{dist}(x,\partial\Omega)>\delta\}.
\end{equation}

Let us recall Example~\ref{ex1} again that the linear equation~\eqref{1u-eq} with the standard Dirichlet boundary conditions on $x=L$ and $x=R$ has a unique solution. However,  under the multi-point boundary condition~\eqref{1-Db}, the existence or uniqueness of solutions to  \eqref{1u-eq} may not hold, where $\lambda$ plays a crucial role in this issue. Based on Example~\ref{ex1}(II), the main motivation is to explore the existence of solutions to \eqref{equ}--\eqref{bdu} with sufficiently large $\lambda$. We shall focus mainly on $f(x,s)$ as follows~({\color{black} see} \cite{BS1984,L2016,Lee2020,S1993}):
\begin{itemize}
    \item[\bf(f1).] For each $s\in\mathbb{R}$, $f(\,\cdot\,,s)\in\text{C}^{0,\tau}(\overline{\Omega};\mathbb{R})$.
    \item[\bf(f2).] For each $x\in\overline{\Omega}$, $f(x,\,\cdot\,)\in\text{C}^{1}(\mathbb{R};\mathbb{R})$ satisfies $\frac{\partial{f}}{\partial{s}}\geq\theta_0$ for some positive constant $\theta_0$, and there uniquely exists~$h\in\text{C}^{2,\tau}(\overline{\Omega};\mathbb{R})$ such that
    \begin{equation}\label{fh}
        f(x,h(x))=0.
    \end{equation}
\end{itemize}
An example for (f1)--(f2) is $f(x,s)=a(x)(s-h(x))$ with a positive function $a\in\text{C}^{0,\tau}(\overline{\Omega};(0,\infty))$.

For the boundary condition~\eqref{bdu}, we focus on the case that $\mathfrak{B}:\mathbb{R}^{m+1}\to\mathbb{R}$ and $\Phi:\mathbb{R}\to\mathbb{R}$ are at least locally Lipschitz continuous in each variable. For simplicity, it suffices to consider the following three cases for $\mathfrak{B}$ and $\Phi$:
\begin{itemize}
    \item[\bf {\color{black} (b1)}.] Both $\mathfrak{B}:\mathbb{R}^{m+1}\to\mathbb{R}$ and $\Phi:\mathbb{R}\to\mathbb{R}$ are globally Lipschitz continuous. To be specific, we assume that there exists $\ell>0$ such that 
 \begin{equation}\label{lip-b}
\begin{cases}
|\mathfrak{B}(s_1,...,s_{m+1})-\mathfrak{B}(t_1,...,t_{m+1})|\leq\ell\displaystyle\sum_{j=1}^{m+1}|s_j-t_j|,\\
|\Phi(s_1)-\Phi(t_1)|\leq\ell|s_1-t_1|,
\end{cases}   
\end{equation}
for all $s_j,t_j\in\mathbb{R}$. 
    \item[\bf {\color{black} (b2)}.] $\mathfrak{B}:\mathbb{R}^{m+1}\to\mathbb{R}$ and $\Phi:\mathbb{R}\to\mathbb{R}$ are only locally Lipschitz continuous.
    \item[\bf {\color{black} (b3)}.] Either $\mathfrak{B}:\mathbb{R}^{m+1}\to\mathbb{R}$ or $\Phi:\mathbb{R}\to\mathbb{R}$ is only locally Lipschitz continuous.
\end{itemize}

For {\color{black} (b1)}, classical examples include $\mathfrak{B}=\sum_{j=1}^m\beta_j\ue(\xi_j)$ and $\mathfrak{B}=\int_{\Omega}w(y)\ue(y)\text{d}y$ with $\Phi(s)=s$, as mentioned previously. Both cases {\color{black} (b2)} and {\color{black} (b3)} find extensive applications in various practical situations, including polynomials, rational functions with positive denominator polynomials, and more. Therefore, it is important to emphasize that this study covers a wider spectrum of nonlocal-type boundary conditions. It is expected that locally Lipschitz continuous functions $\mathfrak{B}$ and $\Phi$ are  applicable to the practical models, and contribute to applied mathematics.


\subsection*{\bf  A fixed point argument based on asymptotic analysis.} 
 The main contribution of this work is the existence issue of \eqref{equ}--\eqref{bdu}. Consider a close relation with the equation
\begin{align}\label{eqv}
    \begin{cases}
 -\nabla\cdot(D(x)\nabla\ve(x))+{\lambda}f(x,\ve(x))=0,&\quad\,x\in\Omega,\vspace{3pt}\\  
 \ve(x)=g(x)+\mu,&\quad\,x\in\partial\Omega,
    \end{cases}
\end{align}
 where the parameter $\mu$ corresponds to the nonlocal part $\mathfrak{B}$ of \eqref{bdu}. It is known that\footnote{Note that equation \eqref{eqv} has a corresponding energy functional $E_{\lambda,\mu}[v]=\int_{\Omega}\left(\frac{D}{2}|\nabla{v}|^2+{\lambda}F(x,v)\right)\dx$ over  the space $\mathcal{H}=\{v\in\text{H}^1(\Omega):\,v-g-\mu\in\text{H}_0^1(\Omega)\}$, where we notice that  $F(x,t):=\int_0^tf(x,s)\,\dss\geq-\frac{\theta_0}{2}\max_{\overline{\Omega}}h^2$ is strictly convex  to $t$ and has a finite infimum (by (f2)). Hence, $E_{\mu,\lambda}[v]$ is a strictly convex functional and we can apply the direct method in the calculus of variations \cite{GT1983} to obtaining that $E_{\mu,\lambda}$ has a unique minimizer $\ve^*$ over $\mathcal{H}$. Since $\Omega$ is bounded with smooth boundary~$\partial\Omega$, $D\in\mathrm{C}^{1,\tau}(\overline{\Omega};(0,\infty))$ and $g\in\mathrm{C}^{0,\tau}(\partial\Omega;\mathbb{R})$, the regularity theory for elliptic equations implies that $\ve^*\in\mathrm{C}^{2,\tau}(\overline{\Omega};\mathbb{R})$ is a classical solution of \eqref{eqv}. The uniqueness of \eqref{eqv} follows immediately from the fact that $f(x,s)$ is strictly increasing to $s$.}  under (f1) and (f2), for each $(\lambda,\mu)\in(0,\infty)\times\mathbb{R}$, equation~\eqref{eqv} has a unique solution~$\ve\in\mathrm{C}^{2,\tau}(\overline{\Omega};\mathbb{R})$. As a consequence, for the existence of solutions $\ue$ to \eqref{equ}--\eqref{bdu}, it suffices to establish a mapping $\boldsymbol{\mathsf{T}}_{\lambda}:\mathbb{R}\to\mathbb{R}$  to connect the equation~\eqref{equ}--\eqref{bdu} with the equation~\eqref{eqv}:
 \begin{equation}\label{map-t}   \boldsymbol{\mathsf{T}}_{\lambda}(\mu):=\mathfrak{B}\big(\ve(\xi_1),...,\ve(\xi_m),\int_{\Omega}w(y)\Phi(\ve(y))\text{d}y\big).
\end{equation} 
The definition of \eqref{map-t} is valid for every $\lambda>0$ and $\mu\in\mathbb{R}$, as the corresponding equation~\eqref{eqv} possesses a unique solution $\ve$. Furthermore, for $\lambda>0$, we have that: 
\begin{itemize}
    \item if $\mu:=\mu(\lambda)\in\mathbb{R}$ depending on $\lambda$ satisfies \begin{equation}\label{mubeta} \mu=\mathfrak{B}\big(\ve(\xi_1),...,\ve(\xi_m),\int_{\Omega}w(y)\Phi(\ve(y))\text{d}y\big), 
 \end{equation}  
 then $\ve$ is a solution of~\eqref{equ}--\eqref{bdu};
    \item if \eqref{equ}--\eqref{bdu} has a solution $\ue$, one sets 
    \begin{equation*}     \mu^{\star}:=\mu^{\star}(\lambda)=\mathfrak{B}\big(\ue(\xi_1),...,\ue(\xi_m),\int_{\Omega}w(y)\Phi(\ue(y))\text{d}y\big),
    \end{equation*}
     then by the uniqueness of equation~\eqref{eqv} there holds $v_{\lambda,\mu^{\star}}=\ue$.
\end{itemize}
As a consequence, \eqref{equ}--\eqref{bdu} has a solution $\ue$ if and only if there exists $\mu:=\mu(\lambda)$ (depending on $\lambda$) satisfies the implicit form \eqref{mubeta} such that $\ue=v_{\lambda,\mu(\lambda)}$. Accordingly, the existence of \eqref{equ}--\eqref{bdu} is equivalently converted to the existence of fixed points of a mapping $\boldsymbol{\mathsf{T}}_{\lambda}:\mathbb{R}\to\mathbb{R}$ defined by \eqref{map-t}.

 In Theorem~\ref{m-thm} ({\color{black} see} Section~\ref{sec-fa}), we will show that under \eqref{lip-b},  as $\lambda>0$ is sufficiently large, $\boldsymbol{\mathsf{T}}_{\lambda}$ has a fixed point, and equation~\eqref{equ}--\eqref{bdu} has a  solution~$\ue$. In particular, when both $\mathfrak{B}:\mathbb{R}^{m+1}\to\mathbb{R}$ and $\Phi:\mathbb{R}\to\mathbb{R}$ are globally Lipschitz continuous, we ensure the uniqueness of the fixed point of $\boldsymbol{\mathsf{T}}_{\lambda}$ as $\lambda\gg1$. This guarantees the existence of a unique solution $\ue$ to equation~\eqref{equ}--\eqref{bdu}. The main idea revolves around establishing the asymptotic estimates for the unique solution~$\ve$ of \eqref{eqv} with respect to $\lambda\gg1$.



\section{\bf The main result: existence, uniqueness and maximum principles}\label{sec-fa}

To employ the fixed point argument to \eqref{map-t}, we shall establish  useful estimates for solutions of \eqref{eqv} and the mapping~$\boldsymbol{\mathsf{T}}_{\lambda}$ defined by \eqref{map-t} with respect to $\lambda\gg1$; see Section~\ref{sec-vt}.  This allows us to establish the following theorem.
\begin{theorem}\label{m-thm}
For given $\xi_j\in\Omega$,  $j=1,...,m$, we define $\mathfrak{m}(\boldsymbol{\xi})$ in \eqref{mxi}, and assume that $f(x,s)$ satisfies {\text(f1)--(f2)}. Then, for $D\in\mathrm{C}^{1,\tau}(\overline{\Omega};(0,\infty))$, $g\in\mathrm{C}^{0,\tau}(\partial\Omega;\mathbb{R})$ and $w\in\text{L}^1(\Omega)$, we consider three cases {\color{black} (b1)}--{\color{black} (b3)} for $\mathfrak{B}:\mathbb{R}^{m+1}\to\mathbb{R}$ and $\Phi:\mathbb{R}\to\mathbb{R}$.
\begin{itemize}
    \item[(I)] If both $\mathfrak{B}:\mathbb{R}^{m+1}\to\mathbb{R}$ and $\Phi:\mathbb{R}\to\mathbb{R}$ are globally Lipschitz continuous (i.e., 
 the case {\color{black} (b1)}), then there exists $\lambda_0>0$ depending on $\ell$ and $\mathfrak{m}(\boldsymbol{\xi})$ such that for each $\lambda>\lambda_0$, \eqref{equ}--\eqref{bdu} has a unique solution  $\ue\in\mathrm{C}^{2,\tau}(\overline{\Omega};\mathbb{R})$. Moreover,
     \begin{equation}\label{m-ax}
  \begin{aligned}   &\lim_{\lambda\to\infty}\left(\max_{x\in\overline{\Omega_{\delta}}}|u_{\lambda}(x)-h(x)|+\max_{x\in\partial\Omega}\left|u_{\lambda}(x)-g(x)-\boldsymbol{\mathsf{B}}[h]\right|\right)=0,
\end{aligned}       
     \end{equation}
     where $\Omega_{\delta}$ is defined by \eqref{i-O}, and
     \begin{equation*}
     \boldsymbol{\mathsf{B}}[h]:=\mathfrak{B}\big(h(\xi_1),...,h(\xi_m),\int_{\Omega}w(y)\Phi(h(y))\text{d}y\big).    
     \end{equation*}
    \item[(II)] If at least one of $\mathfrak{B}:\mathbb{R}^{m+1}\to\mathbb{R}$ or $\Phi:\mathbb{R}\to\mathbb{R}$ is only locally Lipschitz continuous (i.e., the cases {\color{black} (b2)} and {\color{black} (b3)}), then there 
exists $\lambda_0^*>0$ such that as $\lambda>\lambda_0^*$, \eqref{equ}--\eqref{bdu} has at least one solution.
\end{itemize}
\end{theorem}

For (II), we can only verify the existence of \eqref{equ}--\eqref{bdu}. Under {\color{black} (b2)}, this equation may have multiple solutions as $\lambda>0$ is sufficiently large (see Example~\ref{ex2} and the observation from Remarks \ref{rk34} and \ref{rk37}).

As a practical application of Theorem~\ref{m-thm}(I), we further discuss the maximum principles of equation~\eqref{equ} with two type boundary conditions. The first one is the multi-point boundary condition
\begin{equation}\label{mbm}  \ue(x)=g(x)+\sum_{j=1}^m\beta_j\ue(\xi_j),\quad\,x\in\partial\Omega,
\end{equation}
which corresponds to the boundary condition~\eqref{bdu} with $\mathfrak{B}=\sum_{j=1}^m\beta_j\ue(\xi_j)$ and $\beta_j\neq0$, $j=1,...,m$. The other one is the integral boundary condition
\begin{equation}\label{bdi}
\ue(x)=g(x)+\int_{\Omega}w(y)\ue(y)\text{d}y,\quad\,x\in\partial\Omega,
\end{equation}
corresponding to the boundary condition~\eqref{bdu} with $\mathfrak{B}=\int_{\Omega}w(y)\ue(y)\text{d}y$. For $\xi_j\in\Omega$ and $h\in\text{C}^{2,\tau}(\overline{\Omega};\mathbb{R})$ satisfying \eqref{fh}, let $g$ and coefficients $\beta_j$'s in the boundary condition~\eqref{mbm} satisfy 
\begin{equation}\label{gh-m}
\min_{\overline{\Omega}}h-\min_{\partial\Omega}g<\sum_{j=1}^m\beta_jh(\xi_j)<\max_{\overline{\Omega}}h-\max_{\partial\Omega}g.
\end{equation}
Then by \eqref{m-ax}, we have
\begin{equation}\label{m-gh}
    \lim_{\lambda\to\infty}(\min_{x\in\overline{\Omega}}\ue,\max_{x\in\overline{\Omega}}\ue)=(\min_{x\in\overline{\Omega}}h,\max_{x\in\overline{\Omega}}h).
\end{equation}
Similarly, under the boundary condition~\eqref{bdi}, \eqref{m-gh} holds when $g$ and $w$ satisfy
\begin{equation}\label{gh-i}
\min_{\overline{\Omega}}h-\min_{\partial\Omega}g<\int_{\Omega}w(y)h(y)\text{d}y<\max_{\overline{\Omega}}h-\max_{\partial\Omega}g,
\end{equation}
This motivates us to state maximum principles for solutions of  \eqref{equ} with the boundary condition~\eqref{mbm} and the boundary condition~\eqref{bdi}, respectively, as $\lambda$ is sufficiently large.  In this context, the relation between $\sum_{j=1}^m\beta_jh(\xi_j)$ (or $\int_{\Omega}w(y)\ue(y)\text{d}y$) and the maximum/minimum values of $h$ over $\overline{\Omega}$  plays a crucial role. More precisely, by \eqref{gh-m}--\eqref{gh-i}, we have the following result.
\begin{corollary}[\bf  Maximum principle]\label{cor1} Under the same hypotheses as in Theorem~\ref{m-thm}, we assume $g\geq0$ on $\partial\Omega$ $(g\leq0$ on $\partial\Omega$, {respectively}$)$. Then there exists a positive constant $\bar{\lambda}$ relying primarily on $\mathfrak{m}(\boldsymbol{\xi})$, $h$ and $g$ such that for each $\lambda>\bar{\lambda}$, we have the following maximum principle.
\begin{itemize}
    \item[(i)] If $\beta_j$'s and $\xi_j$'s satisfy  \begin{equation*}
\sum_{j=1}^m\beta_jh(\xi_j)>\max_{\overline{\Omega}}h\quad(\sum_{j=1}^m\beta_jh(\xi_j)<\min_{\overline{\Omega}}h,\,\,\text{respectively}),
    \end{equation*} 
 the solution~$\ue$ of \eqref{equ} with the boundary condition~\eqref{mbm} obeys
\begin{equation}\label{max-mi}
\max_{\overline{\Omega}}\ue=\max_{\partial\Omega}\ue\quad(\min_{\overline{\Omega}}\ue=\min_{\partial\Omega}\ue,\,\,\text{respectively}).
\end{equation}
    \item[(ii)] If $w\in\text{L}^1(\Omega)$ satisfies \begin{equation*}
\int_{\Omega}w(y)h(y)\text{d}y>\max_{\overline{\Omega}}h\quad(\int_{\Omega}w(y)h(y)\text{d}y<\min_{\overline{\Omega}}h,\,\,\text{respectively}),
    \end{equation*} 
 the solution~$\ue$ of \eqref{equ} with the boundary condition~\eqref{mbm} obeys \eqref{max-mi}.
\end{itemize}
 
\end{corollary}

The proof of Theorem~\ref{m-thm} consists of several steps. In Section~\ref{sec-vt} we establish the refined estimates for $\ve$ with respect to $\lambda\gg1$ ({\color{black} see} Proposition~\ref{prop-v}), which plays a crucial role in the property of $\boldsymbol{\mathsf{T}}_{\lambda}$. When both $\mathfrak{B}:\mathbb{R}^{m+1}\to\mathbb{R}$ and $\Phi:\mathbb{R}\to\mathbb{R}$ are globally Lipschitz continuous (see {\color{black} (b1)} in Section~\ref{sec-os}), we will demonstrate that, for sufficiently large $\lambda>0$, the mapping~$\boldsymbol{\mathsf{T}}_{\lambda}:\mathbb{R}\to\mathbb{R}$ possesses a unique fixed point, leading to the uniqueness of \eqref{equ}--\eqref{bdu} with the corresponding $\lambda$. However, when {\color{black} (b2)} or {\color{black} (b3)} is satisfied, we can establish only the existence of at least one fixed point for $\boldsymbol{\mathsf{T}}_{\lambda}:\mathbb{R}\to\mathbb{R}$ with sufficiently large $\lambda$; see Proposition~\ref{prop-t}(ii) for details. The proof of Theorem~\ref{m-thm} and Corollary~\ref{cor1} will be completed in Section~\ref{mp-sec}.

\subsection{\bf Estimates of $\boldsymbol{\ve}$ and $\boldsymbol{\mathsf{T}}_{\lambda}$}\label{sec-vt}
Recall that under (f1) and (f2), equation~\eqref{eqv} has a unique solution~$\ve\in\mathrm{C}^{2,\tau}(\overline{\Omega};\mathbb{R})$. We first establish the following estimates for $\ve$.
\begin{proposition}\label{prop-v}
For $\lambda>0$ and $\mu\in\mathbb{R}$, let $\ve\in\mathrm{C}^{2,\tau}(\overline{\Omega};\mathbb{R})$ be the unique solution of \eqref{eqv}. Under the same hypotheses as in Theorem~\ref{m-thm}, there exist positive constants $\mathcal{C}_0$, $\mathcal{M}_0$, $\mathcal{M}_1$  and $\mathcal{M}_2$ independent of $\lambda$ such that as $\lambda>\mathcal{C}_0$,
\begin{equation}\label{esh}
     \max_{x\in\overline{\Omega_{\delta}}}|v_{\lambda,\mu}(x)-h(x)|\leq\mathcal{M}_0\left(\frac{1}{\lambda}+\left(\max_{\partial\Omega}(|g|+|h|)+|\mu|\right)\exp\left(-\sqrt{\lambda}\mathcal{M}_1\delta\right)\right),
\end{equation}
and for $\mu\geq\widetilde{\mu}$,
\begin{equation}\label{v12}
    0\leq\,v_{\lambda,\mu}(x)-v_{\lambda,\widetilde{\mu}}(x)\leq(\mu-\widetilde{\mu})\exp\left(-\sqrt{\lambda}\mathcal{M}_2\mathrm{dist}(x,\partial\Omega)\right),\quad\,x\in\overline{\Omega},
\end{equation}
where $\Omega_{\delta}$ is defined in Theorem~\ref{m-thm}.
\end{proposition}
\begin{remark}\label{rk34}
Note that $\mu$ corresponds to the nonlocal term in \eqref{bdu} and depends on $\lambda$. It is interesting to observe that for the case $\delta=\lambda^{-\alpha}$ with $\alpha\in(0,\frac12)$, if $\mu$ satisfies 
\begin{equation*}
|\mu|\xrightarrow{\lambda\to\infty}\infty\quad\text{and}\quad|\mu|\exp\left(-\lambda^{\frac12-\alpha}\mathcal{M}_1\right)\xrightarrow{\lambda\to\infty}0,   
\end{equation*}
then by the boundary condition of $\ve$ and \eqref{esh}, we have
\begin{equation*}
\min_{\partial\Omega}|\ve|\xrightarrow{\lambda\to\infty}\infty\quad\text{and}\quad\max_{\overline{\Omega_{\lambda^{-\alpha}}}}|v_{\lambda,\mu}-h|\xrightarrow{\lambda\to\infty}0.
\end{equation*}
From this viewpoint, it is conjectured that when $\mathfrak{B}:\mathbb{R}^{m+1}\to\mathbb{R}$ and $\Phi:\mathbb{R}\to\mathbb{R}$ are locally Lipschitz continuous ({\color{black} see} {\color{black} (b2)}), the solutions to equation~\eqref{equ}--\eqref{bdu} may exhibit asymptotic blow-up behavior in a narrow region near the boundary $\partial\Omega$ as $\lambda$ tends to infinity. 
\end{remark}
\begin{proof}[Proof of Proposition~\ref{prop-v}]
By (f2), applying the standard comparison theorem to \eqref{eqv} implies
\begin{equation}\label{max-p}
    \min\left\{\min_{\partial\Omega}g+\mu,\min_{\overline{\Omega}}h\right\}\leq\ve(x)\leq\max\left\{\max_{\partial\Omega}g+\mu,\max_{\overline{\Omega}}h\right\},\quad\,x\in\overline{\Omega}.
\end{equation}
To prove \eqref{esh}, we set
\begin{equation*}
    V(x):=v_{\lambda,\mu}(x)-h(x),
\end{equation*}
where the subscript of $V_{\lambda,\mu}$ is dropped for a sake of simplicity. By \eqref{fh}, \eqref{eqv} and \eqref{max-p}, one may check that
\begin{equation}\label{VV}
    \begin{aligned}   &\,\,\nabla\cdot\left(D(x)\nabla\left(V^2(x)-\frac{M_0^2}{\lambda^2\theta_0^2}\right)\right)=2\nabla\cdot\left(D(x)V(x)\nabla V(x)\right)\\ &\qquad\geq2V(x)\nabla\cdot(D(x)\nabla{V(x)})\\    &\qquad=2V(x)\left[\lambda\left(f(x,V(x)+h(x))-f(x,h(x))\right)-\nabla\cdot(D(x)\nabla{h(x)})\right]\\ &\qquad\geq2\lambda\theta_0V^2(x)-2V(x)\nabla\cdot\left(D(x)\nabla{h(x)}\right)\\
&\qquad\geq2\lambda\theta_0V^2(x)- 2M_0V(x)\\
&\qquad\geq\lambda\theta_0\left(V^2(x)-\frac{M_0^2}{\lambda^2\theta_0^2}\right)\qquad\text{in}\,\,\Omega,
    \end{aligned}
\end{equation}
where $M_0:=\max_{\overline{\Omega}}|\nabla\cdot(D\nabla{h})|$ and we have used (f2) to obtain 
\begin{equation*}
 s\left[f(x,s+h(x))-f(x,h(x))\right]\geq\theta_0s^2,\quad\forall\,s\in\mathbb{R},   
\end{equation*}
 which implies the forth line of \eqref{VV}. The last estimate of \eqref{VV} is obtained by the Cauchy–Schwarz inequality. Note also that $\max_{\partial\Omega}\left(V^2(x)-\frac{M_0^2}{\lambda^2\theta_0^2}\right)\leq\max_{\partial\Omega}(\mu+g-h)^2$. Hence, by applying the maximum principle to \eqref{VV}, we have
\begin{equation}\label{e-bd}
 \max_{\overline{\Omega}}\left(V^2(x)-\frac{M_0^2}{\lambda^2\theta_0^2}\right)\leq\max_{\partial\Omega}(\mu+g-h)^2.  
\end{equation}

To prove \eqref{esh}, we shall establish a refined estimate for $V^2(x)-\frac{M_0^2}{\lambda^2\theta_0^2}$ with respect to $\lambda$. Before deriving a differential inequality based on the form of \eqref{VV}, we introduce a key concept for a rigorous argument. {\color{black}  Due to the smoothness of the boundary $\partial\Omega$, we can choose a sufficiently small $\delta^*$ such that the region $\Omega\setminus\overline{\Omega_{\delta^*}}:=\{x\in\Omega\,:\,\text{d}(x):=\text{dist}(x,\partial\Omega)<\delta^*\}$ (see \eqref{i-O})   is free of focal points and the distance function $\text{d}$ is twice continuously differentiable on its closure, i.e.,} $\text{d}\in\text{C}^2(\overline{\Omega}\setminus{\Omega_{\delta^*}})$~({\color{black} see} \cite{CR2015}).
Hence, we first deal with the estimate of $V^2(x)-\frac{M_0^2}{\lambda^2\theta_0^2}$ in the subdomain $\Omega\setminus\overline{\Omega_{\delta^*}}$ of $\Omega$. By \eqref{VV}--\eqref{e-bd} we define an auxiliary function 
\begin{equation}\label{w-V}
    \widetilde{V}(x)=\max_{\partial\Omega}(\mu+g-h)^2\exp\left({-K_0\sqrt{\lambda}\mathsf{d}^*(x)}\right),\quad\,x\in\overline{\Omega}\setminus{\Omega_{\delta^*}},
\end{equation}
where 
\begin{equation*}
    \mathsf{d}^*(x):=\text{dist}\left(x,\partial(\Omega\setminus\overline{\Omega_{\delta^*}})\right)
\end{equation*}
and $K_0$ is a positive constant to be determined later.

Keeping $|\nabla\mathsf{d}^*|=1$ in mind and making appropriate manipulations, we have
\begin{equation}\label{mkv}
  \begin{aligned}   &\nabla\cdot\left(D(x)\nabla\widetilde{V}(x)\right) \\
&\,=\left[\lambda{K_0^2D(x)}-\sqrt{\lambda}K_0\nabla\cdot(D(x)\nabla\mathsf{d}^*(x))\right]\widetilde{V}(x)\\   &\,\leq{M_1}\left(\lambda{K_0^2}+\sqrt{\lambda}K_0\right)\widetilde{V}(x)\quad\text{in}\,\,\Omega\setminus\overline{\Omega_{\delta^*}},
\end{aligned}  
\end{equation}
where 
\begin{equation}\label{m11}
  M_1=\max\left\{\max_{\overline{\Omega}\setminus{\Omega_{\delta^*}}}D,\max_{\overline{\Omega}\setminus{\Omega_{\delta^*}}}\left|\nabla\cdot(D\nabla\mathsf{d}^*)\right|\right\}  
\end{equation}
 is finite since $\Omega$ is a bounded smooth domain in $\mathbb{R}^N$.

Comparing the right-hand sides of \eqref{VV} with \eqref{mkv}, we shall determine a positive constant~$K_0$ such that as $\lambda>0$ is sufficiently large, there holds
\begin{equation}\label{1028}
  M_1\left(\lambda{K_0^2}+\sqrt{\lambda}K_0\right)\leq\lambda\theta_0.  
\end{equation}
 To this end, one may set
\begin{equation}\label{k024}
 K_0=\sqrt{\frac{\theta_0}{2M_1}}.
\end{equation}
This along with \eqref{1028} gives a lower bound~  $\lambda\geq\left(\frac{M_1K_0}{\theta_0-M_1K_0^2}\right)^2=\frac{2M_1}{\theta_0}$. As a consequence, by  \eqref{VV}--\eqref{e-bd} and \eqref{mkv}, we verify 
 \begin{align}\label{wud}
 \begin{cases}
   \ds\nabla\cdot\left(D(x)\nabla\left(V^2(x)-\frac{M_0^2}{\lambda^2\theta_0^2}-\widetilde{V}(x)\right)\right)\\
\ds\qquad\qquad\qquad\geq  \lambda\theta_0\left(V^2(x)-\frac{M_0^2}{\lambda^2\theta_0^2}-\widetilde{V}(x)\right)\quad\text{in}\,\,\Omega\setminus\overline{\Omega_{\delta^*}},\vspace{3pt}\\
  \ds\, V^2(x)-\frac{M_0^2}{\lambda^2\theta_0^2}-\widetilde{V}(x)\leq0\quad\text{on}\,\,\partial(\Omega\setminus\overline{\Omega_{\delta^*}}).
 \end{cases}
 \end{align}
Applying the maximum principle to \eqref{wud}, we arrive at
\begin{equation*}
 V^2(x)-\frac{M_0^2}{\lambda^2\theta_0^2}\leq\widetilde{V}(x)\,\,\text{in}\,\, \overline{\Omega}\setminus{\Omega_{\delta^*}}, \,\,\text{as}\,\,\lambda\geq\frac{2M_1}{\theta_0}.  
\end{equation*}
Furthermore, we can choose  $\delta^{**}\in(0,\frac{\delta^*}{2})$ such that $\mathsf{d}^*$ verifies
\begin{equation*}
\mathsf{d}^*(x)=\text{dist}(x,\partial\Omega),\quad\forall\,x\in\overline{\Omega}\setminus\Omega_{\delta^{**}}.
\end{equation*}
This implies
\begin{equation}\label{2024}
\begin{aligned}
&\max_{\overline{\Omega}\setminus{\Omega_{\delta^{**}}}}\left( V^2(x)-\frac{M_0^2}{\lambda^2\theta_0^2}\right)\\
&\quad\leq\max_{\partial\Omega}(\mu+g-h)^2\exp\left({-K_0\sqrt{\lambda}\mathsf{d}^*(x)}\right)\,\,\text{as}\,\,\lambda\geq\frac{2M_1}{\theta_0}. 
\end{aligned} 
\end{equation}

We now need to estimate $\max_{x\in\overline{\Omega_{\delta^{**}}}}\left(V^2(x)-\frac{M_0^2}{\lambda^2\theta_0^2}\right)$ for  $\lambda\geq\frac{2M_1}{\theta_0}$  sufficiently large. Notice first that  
\begin{equation*}
\overline{\Omega_{{\delta^*}}}\subset\overline{\Omega_{\delta^{**}}}\quad\text{and}\quad \partial\Omega_{\delta^{**}}\subset\overline{\Omega}\setminus{\Omega_{\delta^*}}\,\,\text{(by \eqref{i-O})}.  
\end{equation*}
 Thus, by \eqref{VV}, \eqref{w-V}, \eqref{k024} and \eqref{2024}, one may use the maximum principle to obtain 
\begin{equation}\label{d*}
\begin{aligned}   &\max_{x\in\overline{\Omega_{\delta^{**}}}}\left(V^2(x)-\frac{M_0^2}{\lambda^2\theta_0^2}\right)=\max_{x\in\partial{\Omega_{\delta^{**}}}}\left(V^2(x)-\frac{M_0^2}{\lambda^2\theta_0^2}\right)\\
\leq&\,\max_{x\in\partial{\Omega_{\delta^{**}}}}\widetilde{V}(x)=\max_{\partial\Omega}(\mu+g-h)^2\exp\left({-\delta^{**}\sqrt{\frac{\theta_0\lambda}{2M_1}}}\right)\\
\leq&\,\max_{\partial\Omega}(\mu+g-h)^2\exp\left({-\frac{\delta^{**}\text{dist}(x,\partial\Omega)}{\text{diam}(\Omega)}\sqrt{\frac{\theta_0\lambda}{2M_1}}}\right),\qquad\,\text{as}\,\,\lambda\geq\frac{2M_1}{\theta_0},
\end{aligned}    
\end{equation}
where $\text{diam}(\Omega)<\infty$ stands for the diameter of the bounded domain $\Omega$. As a consequence, by combining \eqref{w-V} with \eqref{2024}--\eqref{d*}, we know that there exist $\mathcal{C}_0>\frac{2M_1}{\theta_0}$ and positive constants $\mathcal{M}_0$ and $\mathcal{M}_1$ independent of $\lambda$ such that as $\lambda>\mathcal{C}_0$,
\begin{equation}\label{te1}
\begin{aligned}
   &|v_{\lambda,\mu}(x)-h(x)|=|V(x)|\\ \leq&\,\mathcal{M}_0\left(\frac{1}{\lambda}+\left(\max_{\partial\Omega}(|g|+|h|)+|\mu|\right)\exp\left(-\sqrt{\lambda}\mathcal{M}_1\text{dist}(x,\partial\Omega)\right)\right),\quad\forall\,x\in\overline{\Omega}.    
\end{aligned}
\end{equation}
 Therefore, \eqref{esh} follows from  \eqref{te1}.
 
 It remains to prove \eqref{v12}. By \eqref{eqv}, one may check that, for $\mu\geq\widetilde{\mu}$,
 \begin{equation}\label{vmm}
    \begin{aligned}     &\nabla\cdot\left(D(x)\nabla\left(\ve(x)-v_{\lambda,\widetilde{\mu}}(x)\right)\right)\\
   &\qquad ={\lambda}\left(f(x,\ve(x))-f(x,v_{\lambda,\widetilde{\mu}}(x))\right)\\
     &\qquad ={\lambda}\frac{\partial{f}}{\partial{s}}(x,\underline{v}_{\lambda,\mu,\widetilde{\mu}}(x))\left(\ve(x)-v_{\lambda,\widetilde{\mu}}(x)\right)\quad\text{in}\,\,\Omega,
 \end{aligned}  
 \end{equation}
 and
 \begin{equation}
  \ve(x)-v_{\lambda,\widetilde{\mu}}(x)=\mu-\widetilde{\mu}\geq0\quad\text{on}\quad \partial\Omega, \label{BdV}  
 \end{equation}
 where for each $x\in\Omega$, $\underline{v}_{\lambda,\mu,\widetilde{\mu}}(x)$ lies between $\ve(x)$ and $v_{\lambda,\widetilde{\mu}}(x)$. As a consequence, (f2) and the maximum principle imply $\ve(x)-v_{\lambda,\widetilde{\mu}}(x)\geq0$ in $\overline{\Omega}$. Moreover, since $\frac{\partial{f}}{\partial{s}}(x,\underline{v}_{\lambda,\mu,\widetilde{\mu}}(x))\geq\theta_0>0$, by \eqref{vmm} we have
 \begin{equation}  \nabla\cdot\left(D(x)\nabla\left(\ve(x)-v_{\lambda,\widetilde{\mu}}(x)\right)\right)\geq{\lambda}\theta_0\left(\ve(x)-v_{\lambda,\widetilde{\mu}}(x)\right),\quad\,x\in\Omega.\label{EqV}
 \end{equation}

By \eqref{BdV} and \eqref{EqV}, we can follow the same argument as \eqref{VV}--\eqref{te1} to obtain that if $\mathcal{M}_2>\sqrt{\frac{\theta_0}{2M_1}}$ is sufficiently large, then when $\lambda>\mathcal{C}_0$ for some $\mathcal{C}_0>\frac{2M_1}{\theta_0}$, we have 
\begin{equation*}
 0\leq \ve(x)-v_{\lambda,\widetilde{\mu}}(x)\leq(\mu-\widetilde{\mu})\exp\left(-\mathcal{M}_2\sqrt{\lambda}\text{dist}(x,\partial\Omega)\right)\quad\text{in}\,\,\overline{\Omega},\quad\text{as}\,\,\lambda>\mathcal{C}_0, 
\end{equation*}
where $M_1$ is defined by \eqref{m11}. (For the sake of simplicity, we use the same symbol $\mathcal{C}_0$.) Therefore, we arrive at \eqref{v12} and complete the proof of Proposition~\ref{prop-v}.
\end{proof}
\begin{remark}
When $h(x)\equiv{h_0}$ is a constant-valued function, by the argument of \eqref{w-V}--\eqref{te1} we have that
\begin{equation*}
    | v_{\lambda,\mu}(x)-h_0|\leq\max_{\partial\Omega}|\mu+g-h_0|\exp\left({-\mathcal{M}_2\sqrt{\lambda}\mathrm{dist}(x,\partial\Omega)}\right)\quad\text{in}\quad\overline{\Omega},
\end{equation*}
 as $\lambda>\mathcal{C}_0$.

\end{remark}
By Proposition~\ref{prop-v}, we can verify that the mapping $\boldsymbol{\mathsf{T}}_{\lambda}:\mathbb{R}\to\mathbb{R}$ (see \eqref{map-t}) with sufficiently large $\lambda$ is a contraction. We state the result as follows.
\begin{proposition}\label{prop-t}
For $\lambda>0$, let $\boldsymbol{\mathsf{T}}_{\lambda}$ be defined by \eqref{map-t}. Under the same hypotheses as in Theorem~\ref{m-thm} and Proposition~\ref{prop-v}, we {\color{black}have} that: 
\begin{itemize}
    \item[(i)] If both $\mathfrak{B}:\mathbb{R}^{m+1}\to\mathbb{R}$ and $\Phi:\mathbb{R}\to\mathbb{R}$ are globally Lipschitz continuous (see {\color{black} (b1)}), then there exists a positive constant  $\lambda_0>\mathcal{C}_0$  depending mainly on $\ell$ and $\mathfrak{m}(\boldsymbol{\xi})$ such that for each  $\lambda>\lambda_0$, the mapping $\boldsymbol{\mathsf{T}}_{\lambda}:\mathbb{R}\to\mathbb{R}$ has a unique fixed point $\mu=\mu(\lambda)$  depending on $\lambda$.
      \item[(ii)] If at least one of  $\mathfrak{B}:\mathbb{R}^{m+1}\to\mathbb{R}$ or $\Phi:\mathbb{R}\to\mathbb{R}$ is only locally Lipschitz continuous (see {\color{black} (b2)} and {\color{black} (b3)}), then for each $I_{\Lambda}:=[-\Lambda,\Lambda]$ with   
\begin{equation}\label{1611}
\Lambda>\left|\mathfrak{B}\big(h(\xi_1),...,h(\xi_m),\int_{\Omega}w(y)\Phi(h(y))\text{d}y\big)\right|, 
\end{equation}
      there exists a positive constant  $\lambda_0^*=\lambda_0^*(\Lambda)>\mathcal{C}_0$  depending mainly on $\Lambda$ and $\mathfrak{m}(\boldsymbol{\xi})$ such that for each $\lambda>\lambda_0^*$, the mapping $\boldsymbol{\mathsf{T}}_{\lambda}:I_{\Lambda}\to I_{\Lambda}$ has a unique fixed point $\mu=\mu(\lambda,\Lambda)$ depending on $\lambda$ and $\Lambda$. As a consequence, $\boldsymbol{\mathsf{T}}_{\lambda}:\mathbb{R}\to\mathbb{R}$ has at least one fixed point.
\end{itemize}
\end{proposition}
\begin{remark}\label{rk37}
When either {\color{black} (b2)} or {\color{black} (b3)} is satisfied, we note that if $\limsup_{\Lambda\to\infty}\lambda^*_0(\Lambda)<\infty$, then by Proposition~\ref{prop-t}(ii), for sufficiently large $\lambda$, $\boldsymbol{\mathsf{T}}_{\lambda}:\mathbb{R}\to\mathbb{R}$ has a unique fixed point. However, if $\limsup_{\Lambda\to\infty}\lambda_0^*(\Lambda)=\infty$, then, for any sufficiently large $\lambda$, $\boldsymbol{\mathsf{T}}_{\lambda}:\mathbb{R}\to\mathbb{R}$ may have at least two fixed points.
\end{remark}
\begin{proof}[Proof of Proposition~\ref{prop-t}]
Firstly, we consider the case {\color{black} (b1)}. Then, by \eqref{mxi}, \eqref{lip-b}, \eqref{map-t}, \eqref{esh}--\eqref{v12} and \eqref{w-V}, one may check that, as $\lambda>\mathcal{C}_0$,
\begin{equation}\label{t-map}
\begin{aligned}  &|\boldsymbol{\mathsf{T}}_{\lambda}(\mu_1)-\boldsymbol{\mathsf{T}}_{\lambda}(\mu_2)|\\
\leq&\,\ell\left(\sum_{j=1}^m|v_{\lambda,\mu_1}(\xi_j)-v_{\lambda,\mu_2}(\xi_j)|+\int_{\Omega}|w(y)||\Phi(v_{\lambda,\mu_1}(y))-\Phi(v_{\lambda,\mu_2}(y)|\text{d}y\right)\\
  \leq&\,|\mu_1-\mu_2|\left(\ell\sum_{j=1}^m\exp\left(-\sqrt{\lambda}\mathcal{M}_2\text{dist}(\xi_j,\partial\Omega)\right)\right.\\
&\qquad\qquad\quad\left.+\ell^2\int_{\Omega}|w(y)|\exp\left(-\sqrt{\lambda}\mathcal{M}_2\text{dist}(y,\partial\Omega)\right)\text{d}y\right)\\
  \leq&\,\mathcal{M}_{\lambda}(\ell)|\mu_1-\mu_2|,
\end{aligned}
\end{equation} 
where
\begin{equation}\label{ml-11}
\begin{aligned}
    \mathcal{M}_{\lambda}(\ell):=&\, m\ell\exp\left(-\sqrt{\lambda}\mathcal{M}_2\mathfrak{m}(\boldsymbol{\xi})\right)\\
   &\,+\ell^2\left[\exp\left(-\lambda^{\frac14}\mathcal{M}_2\right)\int_{\Omega_{\lambda^{-\frac14}}}|w(y)|\text{d}y+\int_{\Omega\setminus\overline{\Omega_{\lambda^{-\frac14}}}}|w(y)|\text{d}y\right].  
\end{aligned}
 \end{equation}

Note that $w\in\text{L}^1(\Omega)$ and the Lebesgue measure of ${\Omega\setminus\overline{\Omega_{\lambda^{-\frac14}}}}$ approaches zero as $\lambda\to\infty$. Thus, we have
\begin{equation}\label{abw}
\lim_{\lambda\to\infty}\int_{\Omega\setminus\overline{\Omega_{\lambda^{-\frac14}}}}|w(y)|\text{d}y=0. 
\end{equation}
Note also that $\mathfrak{m}(\boldsymbol{\xi})$ and $\ell$ are independent of $\lambda$. Thus, we obtain $\lim_{\lambda\to\infty}\mathcal{M}_{\lambda}(\ell)=0$. As a consequence, one may choose $\lambda_0>\mathcal{C}_0$ depending mainly on $\ell$ and $\mathfrak{m}(\boldsymbol{\xi})$ such that $\mathcal{M}_{\lambda}(\ell)\leq\frac12$. Along with \eqref{t-map}, we obtain that  as $\lambda>\lambda_0$, $\boldsymbol{\mathsf{T}}_{\lambda}:\mathbb{R}\to\mathbb{R}$ is a contraction mapping with $|\boldsymbol{\mathsf{T}}_{\lambda}(\mu_1)-\boldsymbol{\mathsf{T}}_{\lambda}(\mu_2)|\leq\frac12|\mu_1-\mu_2|$. Since $\mathbb{R}$ is a complete metric space, for each $\lambda>\lambda_0$, by the Banach fixed-point theorem, there uniquely exists $\mu(\lambda)\in\mathbb{R}$ such that $\boldsymbol{\mathsf{T}}_{\lambda}(\mu(\lambda))=\mu(\lambda)$.

For the cases {\color{black} (b2)} and {\color{black} (b3)}, it suffices to focus on the case {\color{black} (b2)} since the same argument can be directly applied to the case {\color{black} (b3)}. Let us arbitrarily select a positive constant $\Lambda$ and set 
\begin{equation}\label{al-16}
\widetilde{\Lambda}=\max\left\{\max_{\overline{\Omega}}|h|,\int_{\Omega}\left|w(y)\Phi(h(y))\right|\text{d}y,\Lambda_1,\Lambda_2\right\},
\end{equation}
where
\begin{equation*}
  \Lambda_1:= \Lambda+\max_{\partial\Omega}|g|\quad\text{and}\quad  \Lambda_2:=\max_{|s|\leq\Lambda_1}|\Phi(s)|\int_{\Omega}|w(y)|\dy.  
\end{equation*}
 Then there exists $\ell^*(\Lambda)>0$ depending on $\Lambda$ such that
\begin{equation}\label{lip-*}
\begin{cases}
|\mathfrak{B}(s_1,...,s_{m+1})-\mathfrak{B}(t_1,...,t_{m+1})|\leq\ell^*(\Lambda)\sum_{j=1}^{m+1}|s_j-t_j|,\quad\text{if}\,\,\,|s_j|,\,|t_j|\leq\widetilde{\Lambda},\vspace{3pt}\\
|\Phi(s)-\Phi(t)|\leq\ell^*(\Lambda)|s-t|,\quad\text{if}\,\,\,|s|,\,|t|\leq\widetilde{\Lambda}.
\end{cases}   
\end{equation}
Let $I_{\Lambda}=[-\Lambda,\Lambda]$. We have the following property for $\boldsymbol{\mathsf{T}}_{\lambda}$.
\begin{lemma}[\bf Endomorphism]\label{lem-t}
For $\Lambda$ satisfying \eqref{1611}, there exists $\lambda^{**}(\Lambda)>0$ depending on $\Lambda$ such that 
\begin{equation}\label{1116} \boldsymbol{\mathsf{T}}_{\lambda}(I_{\Lambda})\subseteq I_{\Lambda}\quad\text{as}\,\,\,\lambda>\lambda^{**}(\Lambda).   
\end{equation}
\end{lemma}
\begin{proof} 
Note that, for $|\mu|\leq\Lambda$, there hold
$\max_{\overline{\Omega}}|\ve|\leq|\mu|+\max_{\partial\Omega}|g|\leq\widetilde{\Lambda}$ and
\begin{equation*}
\left|\int_{\Omega}w(y)\Phi(\ve(y))\text{d}y\right|\leq\Lambda_2\leq\widetilde{\Lambda}.   
\end{equation*}
Thus, by \eqref{al-16} and \eqref{lip-*}, we can use the estimate~\eqref{esh} and follow the similar argument as \eqref{t-map} to obtain, for  $\lambda>\mathcal{C}_0$, that:
\begin{equation}\label{add-16}
  \begin{aligned}
 |\boldsymbol{\mathsf{T}}_{\lambda}(\mu)|\leq&\,\left| \mathfrak{B}\big(\ve(\xi_1),...,\ve(\xi_m),\int_{\Omega}w(y)\Phi(\ve(y))\text{d}y\big)\right.\\
   &\left.\qquad\qquad\qquad-\mathfrak{B}\big(h(\xi_1),...,h(\xi_m),\int_{\Omega}w(y)\Phi(h(y))\text{d}y\big)\right|\\ 
&\,+\left|\mathfrak{B}\big(h(\xi_1),...,h(\xi_m),\int_{\Omega}w(y)\Phi(h(y))\text{d}y\big)\right|\\
\leq&\,\ell^*(\Lambda)\left(\sum_{j=1}^m|\ve(\xi_j)-h(\xi_j)|+\ell^*(\Lambda)\int_{\Omega}|w(y)||\ve(y)-h(y)|\text{d}y\right)\\ 
&\,+\left|\mathfrak{B}\big(h(\xi_1),...,h(\xi_m),\int_{\Omega}w(y)\Phi(h(y))\text{d}y\big)\right|\\
\leq&\,\widetilde{\mathcal{M}}_{\lambda}(\ell^*(\Lambda))+\left|\mathfrak{B}\big(h(\xi_1),...,h(\xi_m),\int_{\Omega}w(y)\Phi(h(y))\text{d}y\big)\right|,
\end{aligned}  
\end{equation}
where the upper bound $\widetilde{\mathcal{M}}_{\lambda}(\ell^*(\Lambda))$ is verified by \eqref{mxi} and \eqref{esh} with $|\mu|\leq\Lambda$:
\begin{align*}
\widetilde{\mathcal{M}}_{\lambda}(\ell^*(\Lambda)):=&\,\ell^*(\Lambda)m \mathcal{M}_0\left(\frac{1}{\lambda}+\left(\max_{\partial\Omega}(|g|+|h|)+\Lambda\right)\exp\left(-\sqrt{\lambda}\mathcal{M}_1\mathfrak{m}(\boldsymbol{\xi})\right)\right)   \\
&\,+(\ell^*(\Lambda))^2\left[\exp\left(-\lambda^{\frac14}\mathcal{M}_2\right)\int_{\Omega_{\lambda^{-\frac14}}}|w(y)|\text{d}y+\int_{\Omega\setminus\overline{\Omega_{\lambda^{-\frac14}}}}|w(y)|\text{d}y\right],
\end{align*}
which is similar to \eqref{ml-11}. One observes that for each $\Lambda$ satisfying \eqref{1611}, there hold $\lim_{\lambda\to\infty}\widetilde{\mathcal{M}}_{\lambda}(\ell^*(\Lambda))=0$. Hence, by \eqref{add-16}, there exists $\lambda^{**}(\Lambda)>\mathcal{C}_0$ depending on $\Lambda$ such that as $\lambda>\lambda^{**}(\Lambda)$, \eqref{1116} is satisfied. The proof of Lemma~\ref{lem-t} is thus completed.
\end{proof}

Finally, we want to prove that for each $\Lambda$ satisfying \eqref{1611}, $\boldsymbol{\mathsf{T}}_{\lambda}:I_{\Lambda}\to I_{\Lambda}$ is a contraction mapping provided that $\lambda>\lambda^{**}(\Lambda)$ is sufficiently large. Indeed, by \eqref{lip-*}, one may follow the same argument as \eqref{t-map} to obtain $|\boldsymbol{\mathsf{T}}_{\lambda}(\mu_1)-\boldsymbol{\mathsf{T}}_{\lambda}(\mu_2)|\leq\mathcal{M}^*_{\lambda,\Lambda}|\mu_1-\mu_2|$ for $\mu_1,\,\mu_2\in I_{\Lambda}$. Here $\mathcal{M}^*_{\lambda,\Lambda}$ corresponds to $\mathcal{M}_{\lambda}(\ell)$ in \eqref{ml-11} with $\ell=\ell^*(\Lambda)$ which depends on $\Lambda$. Since for fixed $\Lambda$ we have $\lim_{\lambda\to\infty}\mathcal{M}^*_{\lambda,\Lambda}=0$, there exists $\lambda^*=\lambda^*(\Lambda)>\lambda^{**}(\Lambda)$ depending on $\mathfrak{m}(\boldsymbol{\xi})$ and $\Lambda$ such that $\boldsymbol{\mathsf{T}}_{\lambda}:I_{\Lambda}\to I_{\Lambda}$ admits
\begin{equation*}
    |\boldsymbol{\mathsf{T}}_{\lambda}(\mu_1)-\boldsymbol{\mathsf{T}}_{\lambda}(\mu_2)|\leq\frac12|\mu_1-\mu_2|,\quad\text{for}\,\,|\mu_1|,\,|\mu_2|\leq\Lambda,\quad\text{as}\,\,\lambda>\lambda^*(\Lambda).
\end{equation*}
Therefore,  as $\lambda>\lambda^*(\Lambda)$, $\boldsymbol{\mathsf{T}}_{\lambda}$ has a unique fixed point in $[-\Lambda,\Lambda]$, and the proof of Proposition~\ref{prop-t} is completed.
\end{proof}

\subsection{Proof of Theorem~\ref{m-thm} and Corollary~\ref{cor1}}\label{mp-sec}
Having Propositions~\ref{prop-v} and \ref{prop-t} in hands, we are now in a position to state the proof of Theorem~\ref{m-thm} as follows.
\begin{proof}[Proof of  Theorem~\ref{m-thm}]
For each fixed $\lambda>\lambda_0$, by \eqref{map-t} and Proposition~\ref{prop-t}(i),  there exists $\mu(\lambda)$ such that
\begin{equation*}
\mu(\lambda)=\mathfrak{B}\big(v_{\lambda,\mu(\lambda)}(\xi_1),...,v_{\lambda,\mu(\lambda)}(\xi_m),\int_{\Omega}w(y)\Phi(v_{\lambda,\mu(\lambda)}(y))\text{d}y\big),   
\end{equation*}
where $v_{\lambda,\mu(\lambda)}\in\mathrm{C}^{2,\tau}(\overline{\Omega};\mathbb{R})$ is the unique solution of \eqref{eqv} corresponding to $\mu=\mu(\lambda)$. Thus, $\ue=v_{\lambda,\mu(\lambda)}$ is a solution of \eqref{equ}--\eqref{bdu}. 

Proposition~\ref{prop-t}(i) also implies the uniqueness of the solution $\ue$ to \eqref{equ}--\eqref{bdu}. To be more specific, suppose on the contrary that there is a $\lambda_{\mathfrak{y}}>\lambda_0$ such that \eqref{equ}--\eqref{bdu} with $\lambda=\lambda_{\mathfrak{y}}$ has at least two distinct solutions $u_{\lambda_{\mathfrak{y}},1}$ and $u_{\lambda_{\mathfrak{y}},2}$, i.e.,
\begin{equation}\label{2025-1}
u_{\lambda_{\mathfrak{y}},1}\not\equiv   u_{\lambda_{\mathfrak{y}},2}.
\end{equation}
Setting 
\begin{equation*}
\mu_i(\lambda_{\mathfrak{y}})=\mathfrak{B}\big(u_{\lambda_{\mathfrak{y}},i}(\xi_1),...,u_{\lambda_{\mathfrak{y}},i}(\xi_m),\int_{\Omega}w(y)\Phi(u_{\lambda_{\mathfrak{y}},i}(y))\text{d}y\big),\quad\,i=1,2,
\end{equation*}
 by the uniqueness of \eqref{eqv} there must hold $v_{\lambda_{\mathfrak{y}},\mu_i(\lambda_{\mathfrak{y}})}=u_{\lambda_{\mathfrak{y}},i}$. In particular, this indicates  
 \begin{equation*}
\mu_i(\lambda_{\mathfrak{y}})=\mathfrak{B}\big(v_{\lambda_{\mathfrak{y}},\mu_i(\lambda_{\mathfrak{y}})}(\xi_1),...,v_{\lambda_{\mathfrak{y}},\mu_i(\lambda_{\mathfrak{y}})}(\xi_m),\int_{\Omega}w(y)\Phi(v_{\lambda_{\mathfrak{y}},\mu_i(\lambda_{\mathfrak{y}})}(y))\text{d}y\big)=\mathsf{T}_{\lambda_{\mathfrak{y}}}(\mu_i(\lambda_{\mathfrak{y}})).   
 \end{equation*}
  By Proposition~\ref{prop-t} we have $\mu_1(\lambda_{\mathfrak{y}})=\mu_2(\lambda_{\mathfrak{y}})$ since $\lambda_{\mathfrak{y}}>\lambda_0$. Hence, 
  \begin{equation*} u_{\lambda_{\mathfrak{y}},1}=v_{\lambda_{\mathfrak{y}},\mu_1(\lambda_{\mathfrak{y}})}=v_{\lambda_{\mathfrak{y}},\mu_2(\lambda_{\mathfrak{y}})}=u_{\lambda_{\mathfrak{y}},2},    
  \end{equation*}
 a contradiction to \eqref{2025-1}. Therefore, we complete the proof of uniqueness. 

Now, for $\lambda>\lambda_0>\mathcal{C}_0$,
we have $\ue=v_{\lambda,\mu(\lambda)}$. Then, by \eqref{mxi}, \eqref{lip-b}, \eqref{mubeta} and \eqref{esh} we have
\begin{equation}\label{wab}
    \begin{aligned}
   |\mu(\lambda)-\boldsymbol{\mathsf{B}}[h]|=&\,\left| \mathfrak{B}\big(v_{\lambda,\mu(\lambda)}(\xi_1),...,v_{\lambda,\mu(\lambda)}(\xi_m),\int_{\Omega}w(y)\Phi(v_{\lambda,\mu(\lambda)}(y))\text{d}y\big)\right.\\
   &\left.\qquad\qquad\qquad-\mathfrak{B}\big(h(\xi_1),...,h(\xi_m),\int_{\Omega}w(y)\Phi(h(y))\text{d}y\big)\right|\\ \leq&\,\ell\left(\sum_{j=1}^m|v_{\lambda,\mu(\lambda)}(\xi_j)-h(\xi_j)|+\ell\int_{\Omega}|w(y)||v_{\lambda,\mu(\lambda)}(y)-h(y)|\text{d}y\right)\\ \leq&\,\mathcal{M}_0\ell{\sum_{j=1}^m}\left(\frac{1}{\lambda}+\left(\max_{\partial\Omega}(|g|+|h|)+|\mu(\lambda)|\right)\exp\left(-\sqrt{\lambda}\mathcal{M}_1 \mathfrak{m}(\boldsymbol{\xi})\right)\right)\\
&+\mathcal{M}_0\ell^2\left(\frac{1}{\lambda}+\max_{\partial\Omega}(|g|+|h|)+|\mu(\lambda)|\right)\\
&\quad\times\left(\int_{\Omega\setminus\overline{\Omega_{\lambda^{-\frac14}}}}|w(y)|\text{d}y+\exp\left(-{\lambda}^{\frac14}\mathcal{M}_1\right)\int_{\Omega_{\lambda^{-\frac14}}}|w(y)|\text{d}y \right).
\end{aligned}
\end{equation}
Moreover, let us use the elementary inequality $|\mu(\lambda)|\leq\left| \mu(\lambda)-\boldsymbol{\mathsf{B}}[h]\right|+\left|\boldsymbol{\mathsf{B}}[h]\right|$ to deal with the last estimate of \eqref{wab} and notice
\begin{equation*}
 \exp\left(-\sqrt{\lambda}\mathcal{M}_1 \mathfrak{m}(\boldsymbol{\xi})\right)\xrightarrow{\lambda\to\infty}0\quad\text{and}\quad \int_{\Omega\setminus\overline{\Omega_{\lambda^{-\frac14}}}}|w(y)|\text{d}y\xrightarrow{\lambda\to\infty}0\,\,\,(\text{see\,\,(\ref{abw})}). 
\end{equation*}
After making appropriate manipulations, we obtain $\mu(\lambda)\xrightarrow{\lambda\to\infty}\boldsymbol{\mathsf{B}}[h]$, and by \eqref{esh} with $\ue=v_{\lambda,\mu(\lambda)}$, we have
\begin{equation*}  \max_{\overline{\Omega_{\delta}}}|u_{\lambda}-h|\xrightarrow{\lambda\to\infty}0\quad\text{and}\quad\max_{\partial\Omega}\left|\ue-g-\boldsymbol{\mathsf{B}}[h]\right|\xrightarrow{\lambda\to\infty}0.
\end{equation*}
This implies \eqref{m-ax} and the proof of Theorem~\ref{m-thm}(I) is thus completed.

By Proposition~\ref{prop-t}(ii), for $\lambda>\lambda_0^*$, $\boldsymbol{\mathsf{T}}_{\lambda}$ has a  fixed point $\mu^*(\lambda)$ (which may not be unique). Hence, $\ue=v_{\lambda,\mu^*(\lambda)}$ is a solution of \eqref{equ}--\eqref{bdu}. This proves Theorem~\ref{m-thm}(II). Therefore, we  complete the proof of Theorem~\ref{m-thm}.
\end{proof}
The proof of Corollary~\ref{cor1} is stated as follows.
\begin{proof}[Proof of Corollary~\ref{cor1}] 
We state the proof of (i) as follows. Note first that $\mathfrak{B}(s_1,...,s_m)=\sum_{j=1}^m\beta_js_j$ is globally Lipschitz continuous. Thus, by Theorem~\ref{m-thm}(I), there exists a constant $\lambda_0>0$ such that the solution~$\ue$ of \eqref{equ} with the boundary condition~\eqref{mbm} is unique as $\lambda>\lambda_0$. It suffices to consider the case where
\begin{equation*}
g\geq0\,\,\text{on} \,\, \partial\Omega,\,\,\text{and}\,\,\sum_{j=1}^m\beta_jh(\xi_j)>\max_{\overline{\Omega}}h.  
\end{equation*}
Setting   
\begin{equation*}
\sigma=\frac12\left(\sum_{j=1}^m\beta_jh(\xi_j)-\max_{\overline{\Omega}}h\right)>0,
\end{equation*}
we shall prove $\max_{\overline{\Omega}}\ue=\max_{\partial\Omega}\ue$ provided that $\lambda$ is larger than a positive constant depending mainly on $\sigma$. Suppose on the contrary that there exists a sequence $\{\lambda_k\}_{k\in\mathbb{N}}$ with $\lambda_0<\lambda_k\xrightarrow{k\to\infty}\infty$, and for each $\lambda_k$, there exists 
 an interior point $x_k\in\Omega$ such that $u_{\lambda_k}(x_k)=\max_{\overline{\Omega}}u_{\lambda_k}$. Then by \eqref{equ} one obtains 
 \begin{equation*}
  \lambda_kf(x_k,u_{\lambda_k}(x_k))=\nabla\cdot(D(x)\nabla\ue(x))|_{x=x_k}\leq0.   
 \end{equation*}
 Hence, we have $f(x_k,u_{\lambda_k}(x_k))\leq0$. Along with (f2) one immediately arrives at 
 \begin{equation*}
\max_{\overline{\Omega}}u_{\lambda_k}=u_{\lambda_k}(x_k)\leq{h(x_k)}<\sum_{j=1}^m\beta_jh(\xi_j)-\sigma.  
\end{equation*}

On the other hand, for each $x\in\partial\Omega$, we have
\begin{equation*}
\lim_{k\to\infty}u_{\lambda_k}(x)=g(x)+\sum_{j=1}^m\beta_jh(\xi_j)\geq\sum_{j=1}^m\beta_jh(\xi_j).    
\end{equation*}
 (Note that $\boldsymbol{\mathsf{B}}[h]=\sum_{j=1}^m\beta_jh(\xi_j)$ here.) Hence, as $\lambda_k\gg1$, we arrive at
 \begin{equation*}  \max_{\partial\Omega}u_{\lambda_k}>\sum_{j=1}^m\beta_jh(\xi_j)-\sigma>\max_{\overline{\Omega}}u_{\lambda_k}
 \end{equation*}
  which leads to a contradiction. Therefore, we obtain that there exists a $\bar{\lambda}>0$ depending on $\sigma$ such that as $\lambda>\bar{\lambda}$, there holds
$\max_{\overline{\Omega}}\ue=\max_{\partial\Omega}\ue$. 

Following the same argument to the case where $g\leq0$ on $\partial\Omega$, and $\sum_{j=1}^m\beta_jh(\xi_j)<\min_{\overline{\Omega}}h$, we can obtain $\min_{\overline{\Omega}}\ue=\min_{\partial\Omega}\ue$,  and thus prove (i).

For (ii), i.e., under the boundary condition~\eqref{bdi}, we know that equation~\eqref{equ} with sufficiently large $\lambda$ has a unique solution $\ue$ satisfying 
\begin{equation*}
\lim_{\lambda\to\infty}\ue(x)=g(x)+\int_{\Omega}w(y)h(y)\text{d}y\,\,\,\text{uniformly\,on\,the\,boundary}\,\,\partial\Omega.    
\end{equation*}
 Hence, by the same argument as (i), we can prove (ii). This completes the proof of Corollary~\ref{cor1}.
\end{proof}

\subsection*{Acknowledgement}
This research was completed during an enriching visit to Professor Chun Liu at the Illinois Institute of Technology (IIT) in the Fall of 2023. The author is deeply grateful to Professor Liu and IIT for providing a conducive and supportive environment that significantly facilitated this work. Sincere gratitude is also extended to the referee for providing valuable feedback that substantially improved the final manuscript. Partial support for this research was provided by the Taiwan Ministry of Science and Technology under grant number 112-2115-M-007-008-MY2.

\end{document}